\documentclass{article}

\usepackage{bbm}
\usepackage[utf8]{inputenc}
\usepackage[english]{babel}
\usepackage{color}
\usepackage{amsmath,amssymb,amsthm,yhmath,amsfonts,mathrsfs,enumitem,mathtools,dsfont}
\usepackage{tikz}
\usetikzlibrary{shapes,arrows,positioning}

\newcommand{\pr}{\mathbb{P}}
\newcommand{\E}{\mathbb{E}}
\renewcommand{\d}{\mathrm{d}}
\newcommand{\erg}{Erd\H os--R\'enyi graph}

\newcommand{\W}{\mathbb{W}}
\newcommand{\1}[1]{{\mathds{1}}_{\left\{#1\right\}}}
\newcommand{\I}{\mathcal{I}}
\renewcommand{\S}{\mathcal{S}}

\newcommand{\st}{\ \ \textit{st.}}

\newtheorem{theorem}{Theorem}
\newtheorem{lemma}{Lemma}
\newtheorem{prop}{Proposition}

\DeclareMathOperator{\cov}{cov}

\title{Local-density dependent Markov processes on graphons with epidemiological applications}

\author{
	D\'aniel Keliger\\
	{\small Department of Stochastics,}\\
	{\small Budapest University of Technology and Economics}\\
	{\small e-mail: perfectumfluidum@gmail.com}\\[3mm]
	Ill\'es Horv\'ath\\
	{\small MTA-BME Information Systems Research Group}\\
	{\small e-mail: horvath.illes.antal@gmail.com}\\[3mm]
	B\'alint Tak\'acs \\
	{\small Department of Applied Analysis and Computational Mathematics}\\
	{\small E\"otv\"os Lor\'and University}\\
	{\small e-mail: takacsbm@caesar.elte.hu}
}

\date{}

\begin{document}
	
	
	\maketitle
	
\begin{abstract}
We investigate local-density dependent Markov processes on a class of large graphs sampled from a graphon, where the transition rates of the vertices are influenced by the states of their neighbors. We show that as the average degree converges to infinity, the evolution of the process in the transient regime converges to the solution of a set of non-local integro-partial differential equations. We also provide rigorous derivation for the epidemic threshold in the case of the Susceptible-Infected-Susceptible (SIS) process on such graphons.  
\end{abstract}

\section{Introduction}

Large random graphs have been used to model complex networks such as communication systems or biological and social populations, where the vertices represent individuals and the edges correspond to interactions between them \cite{easleybook}. Generally, as the network size increases, such networks can be divided into two main classes: in \emph{dense graphs}, the number of interactions of a single vertex scales with the network size, as opposed to \emph{sparse graphs}, where it does not. Dense graph models include the {\erg} and the stochastic block model, while sparse graph models include the configuration model or scale-free networks.

Many different types of processes have been examined on large networks. One of the most important applications is infection spreading models such as the SI (susceptible-infected), SIS (susceptible-infected-susceptible) and SIR (susceptible-infected-recovered) models, but there have been applications for wireless sensor networks \cite{Bruneo2011}, push-pull gossip protocols \cite{Bakhshi2010a}, peer-to-peer software updates \cite{hht2014}, biological systems \cite{Caravagna2011} or molecular biology \cite{Schlicht2008}. A common trait is that as the network size increases, a complete analysis becomes computationally infeasible; instead, mean-field approximations become relevant.

One of the pioneering mean-field results is due to Kurtz \cite{kurtz70,kurtz78}, who examined density dependent Markov population processes. Kurtz proved that as the number of components increases, in the transient regime the process converges to the solution of a system of differential equations \cite{kurtz70}.

Markov population processes can be used to model a wide range of processes, but the density dependent assumption essentially corresponds to a very restrictive assumption on network structure, e.g. it is a complete graph, where interactions can not really be considered local. Nevertheless, the results of Kurtz have been used extensively as an approximation in many scenarios, e.g. chemical reactions and queuing systems \cite{kurtz2011, telek2008}.

For other network structures where the interactions are truly local, results have been sporadic. One special example is the SIR model on the configuration model \cite{volz,volzproof}, with a mean-field limit different from Kurtz.

In this paper, we use \emph{graphons} \cite{lovaszbook}, a relatively recent tool to examine large complex network structures. Graphons typically arise as the limit of dense graph sequences \cite{orbanz}.

The main contributions of the paper are the following: we are going to define local density-dependent Markov processes on finite graphs, identify their corresponding limit processes on graphons, and prove mean-field convergence in the transient regime under mild density conditions. We also observe that the same limit does not hold for truly sparse graphs. Finally, we focus on the SIS model for which we derive the epidemic threshold in the graphon limit.

The rest of the paper is organized as follows. Section \ref{s:related} gives an overview of relevant literature. Section \ref{s:graphons} provides a background for large graphs and graphons. Section \ref{s:setup} describes the necessary setup for local-density dependent Markov processes on large graphs and graphons, and also states the main results. Sections \ref{s:pde} and \ref{s:convproof} provide the rigorous proofs for the results, with Section \ref{s:pde} focusing on regularity properties of the mean-field limit equation on graphons (with Section \ref{s:SIS} dedicated to the SIS model), and Section \ref{s:convproof} focusing on the proof of mean-field convergence. Section \ref{s:concl} concludes the work.

\section{Related works}
\label{s:related}

In this section we review some of the recent works related to dynamical systems on large networks and graphons.

In \cite{graphonrandomwalk} the behavior of  random walks on some large dense graphs is investigated. The underlying assumptions are mild in general: this includes the corresponding graph limit being in some sense connected, and the existences of a uniform lower bound for the degrees. Graphs are generated in a similar fashion to ours with the exception of only considering the dense case. From the linear nature of the random walk it is enough to consider the occupation probabilities for which a system of ODEs can be derived. The paper shows that the limiting object of such system is a well-posed non-local PDE.

In \cite{graphlimitKuramoto} the non-linear heat equation - and as a special case, the Kuramoto model -  is described on large dense graphs. The limit object is also a well-posed non-local PDE for which the rate of convergence is given in the special case when the graphon takes values from $\{0,1\}$. This  case includes systems where particles interact within a finite range scaling with the number of particles. 

In both \cite{graphonSISnoise} and \cite{graphonSIScontrol} the authors study the Susceptible-Infected-Susceptible (SIS) model  on large dense graphs. Both papers investigate a linearised ODE approximation of the SIS process based on meta-population considerations on dense graphs. The main difference is \cite{graphonSISnoise} works with an additive noise term put into the quenched ODE system artificially to study the stability of the epidemic, while \cite{graphonSIScontrol} works with an additional control term instead.

In many infection spreading models, there is a phase transition: if parameter $\beta$ is below a certain critical value $\beta_c$, known as the \emph{epidemic threshold}, the infection will die out rapidly, while if $\beta>\beta_c$, the infection will spread to a positive fraction of the entire population. $\beta_c$ has been computed for certain models such as individual-based mean-field (IBMF) models \cite{threshold2008},  SIS model on scale-free networks \cite{vesp}, while bounds have been proven for the $N$-intertwined mean-field approximation (NIMFA) model \cite{simon2017NIMFA} and for more general finite graph settings \cite{satorras2015}.

Lastly, the authors of \cite{graphonSDE} examine a system of SDE's which interact with each other through a network. They derive a large graph limit for the case when average degree approaches infinity.
 
\section{Graphs and graphons}
\label{s:graphons}

A \emph{graph} is defined as a pair $G = (V,A)$, where $V$ is a
finite set of vertices and $A: V\times V \to \mathbb{R}$ is the adjacency matrix. In this paper, we will use $V=\{1,2,\dots,N\}$ exclusively.

We only consider unweighted, simple, undirected graphs, i.e.~$a_{ij}^N\in\{0,1\}$, $a_{ii}^N=0$ and $a_{ji}^N=a_{ij}^N$ for all $i,j\in V$ (where $A=(a_{ij}^N)_{i,j\in V}$). Equivalently, we could define the graph with the set of edges $E=\{(i,j): a_{ij}^N=1,\,i,j\in V\}$ instead of $A$. The \emph{degree} of a vertex $i$ is
$$d^N(i)=\sum_{j=1}^N a_{ij}^N.$$

A \emph{graphon}, or rather, the \emph{kernel} of a graphon is defined as a measurable symmetric function $W:[0,1]\times [0,1]\to \mathbb{R}$. $\mathcal{W}$ will denote the set of all graphon kernels, while  $\mathcal{W}_0$ denotes graphon kernels that are $[0,1]\times [0,1]\to [0,1]$.

For any graphon kernel $W$, the corresponding integral operator, acting on $f\in L^1 \left([0,1] \right)$ functions is
$$(\W f)(x):=\int_0^1 W(x,y)f(y)\d y.$$

For vector-valued functions $f$, the notation $\W f$ is understood component-wise.

The \emph{degree function} of a graphon is
$$d_W(x):=(\W \mathbf{1})(x)=\int_0^1 W(x,y)\d y,$$
where $\mathbf{1}$ denotes the constant one function.

\subsubsection*{Sampling}

From any graphon kernel $W(x,y)\in\mathcal{W}_0$, we can sample a simple random graph on $N$ vertices. For a given $\W$ and $N$, we do this the following way: first, we generate $U_i^N, i=1,\dots,N$ increasingly ordered independent $U[0,1]$ random variables. Alternatively, using $U_i^N=\frac{i}{N}$ suits our goal as well. Then we set \linebreak$a_{ij}^N=1$ with probability $W(U_i^N,U_j^N)$ conditionally independently for each pair \linebreak$1\leq i<j\leq N$. Finally, $a_{ij}^N=a_{ji}^N$ and $a_{ii}^N=0$ is used. The sample graphon is also called an \emph{empirical graphon}. A sample graphon is usually denoted by $W^N$.

It will be convenient to consider sampling with an extra density parameter $\kappa=\kappa^N>0$ which may depend on $N$; when $\kappa^N$ is given, we set $a_{ij}^N=1$ with probability $\kappa^N W(U_i^N,U_j^N)$ instead.

On the other hand, any graph on $N$ vertices with adjacency matrix $A$ can be interpreted as a graphon the following way. Denoting
\begin{align}
\label{eq:IiNdef}
I_i^N=\left[\frac{i-1}N,\frac{i}N\right)
\end{align}
(with the last interval also including the endpoint $1$), and we set 
$$W(x,y)=a_{ij}^N \quad \textrm{for} \quad x\in I_i^N, y\in I_j^N, i,j=1,\dots,N.$$

\subsubsection*{Norms}

Let $ \mathcal{S}$ be a finite set. The norm $ \| \cdot \| $ for $v \in \mathbb{R}^\mathcal{S} $ is chosen to be the $ L^1$ norm $$ \left \|v \right \|=\sum_{s \in \mathcal{S}}\left |v_i \right |. $$
This does not result in a loss of generality since $\mathcal{S} $ is finite.

The \emph{cut norm} of a kernel is
$$\| W\|_{\square}=\sup_{S,T\subseteq [0,1]}\left| \int_{S\times T}W(x,y)\d x\d y\right|.$$

Actually, we identify kernels $W_1$ and $W_2$ such that $\| W_1-W_2\|_{\square}=0$, and factorize the spaces $\mathcal{W}$ and $\mathcal{W}_0$ accordingly.

For $f:[0,1]\to \mathbb{R}$ functions, the $L^p$-norm is defined as
$$\|f\|_p=\left(\int |f(x)|^p\d x\right)^{1/p},$$
and the corresponding \emph{operator norm} of a kernel operator is defined as
$$\| \W\|_{\textrm{op.},p}=\sup_{f\in L^p[0,1], \|f\|_p=1} \|\W f\|_p.$$

Let $\I$ denote the set of all subintervals of $[0,1]$. The \emph{interval norm} of $f:[0,1]\to \mathbb{R}$ is defined as
$$\|f\|_\I=\sup_{I\in \I}\left|\int_I f(x)\mathrm{d}x \right|.$$

We extend the $L^p$ and the interval norms to vector-valued functions: for $f:[0,1]\to \mathbb{R}^\S$ where $\S$ is a finite set, we denote the components of $f$ by $f_s,\, s\in\S$, and set
\begin{align*}
\|f\|_p=\sum_{s\in\S} \|f_s\|_p,\qquad \qquad
\|f\|_\I=\sum_{s\in\S} \|f_s\|_\I.
\end{align*}
The norms for vector-valued functions is mostly a technicality, as $S$ will be kept fixed throughout.

The \emph{$L^p$-norm} of a kernel is
$$\| W \|_p=\left(\int_{[0,1]^2} |W(x,y)|^p \d x\d y\right)^{1/p}.$$

On $\mathcal{W}_0$, the following norm relations hold:
\begin{align}
\| W\|_{\square}&= \|W\|_1\leq \|W\|_2\leq \|W\|_1^{1/2}\leq 1,\\
\label{eq:op2boxnormequiv}
\| W\|_{\square}&\leq \|\W\|_{\textrm{op},2}\leq \sqrt{8}\| W\|_{\square}^{1/2}.
\end{align}

\subsubsection*{Spectral properties}

Every graphon kernel operator $\W$ has a discrete spectrum with 0 as the only accumulation point, and every nonzero eigenvalue has finite multiplicity. The nonzero eigenvalues are ordered as $|\lambda_1(\W)|\geq |\lambda_2(\W)|\geq \dots$, with corresponding eigenfunctions $f_k(x),k=1,2,\dots$, normalized by $\|f\|_2=1$. $\lambda_1$ is always positive.

The spectral decomposition of the corresponding kernel $W$ is
$$W(x,y)=\sum_{k=1}^\infty \lambda_k(\W)f_k(x)f_k(y).$$

A graphon $W$ is said to have \emph{finite rank} if the spectrum of $\W$ only has a finite number of nonzero eigenvalues \cite{lovaszbook}. In this case, the spectral decomposition is also finite.

Any graphon kernel can be \emph{truncated} to a finite rank $K$ via
$$W_K(x,y)=\sum_{k=1}^K \lambda_k(\W)f_k(x)f_k(y).$$

The truncated $W_K$ may be negative even if the original $W$ is from $\mathcal{W}_0$.

The norms relate to the spectrum via
$$\| W\|_2^2=\sum_{k=1}^\infty |\lambda_k(\W)|^2 \qquad \textrm{and}\qquad \|\W\|_{\textrm{op},2}=\lambda_1(W).$$

We also have the error bound
$$\| W-W_K\|_2^2=\sum_{k=K+1}^\infty |\lambda_k(\W)|^2$$
for the truncated kernel.

The adjacency matrix $A$ of any simple graph on $N$ vertices is real and symmetric, so its spectrum is also real and will be ordered as \linebreak $|\lambda_1(A)|\geq |\lambda_2(A)|\geq \dots \geq \left |\lambda_N(A) \right |$.

If $A^N$ denotes the sample graph of the graphon $\W$ on $N$ vertices, then the spectrum of $A^N$ is the normalized spectrum of $\W$ \cite{randomgraphspectra}, i.e.
$$\lambda_i(\W)=\lim_{N\to\infty}\lambda_i(A^N)/(\kappa_NN).$$

\subsubsection*{Random graphs and limit graphons}

In the \emph{{\erg}} on $N$ vertices, each pair of vertices $i\neq j$ is connected with probability $0<p<1$ (density parameter), independent from other pairs. We assume $G^N$ is undirected and simple, so $a_{ii}^N=0 $ and $a_{ij}^N=a_{ji}^N$, with $\{a_{ij}:i>j\}$ independent Bernoulli($p$) distribution.

The corresponding random graphon $A^N$ is blockwise constant with value $a_{ij}$ on blocks $x \in \left[\frac{i-1}N,\frac{i}N\right), y\in \left[\frac{j-1}N,\frac{j}N\right),\, i,j=1\dots N$.

As $N\to\infty$, the corresponding graphon $W^N$ converges to the graphon $W(x,y)\equiv p$ in $\|.\|_\square$ and $\|.\|_{\textrm{op.},p}$ but not in $\|.\|_{p}$ (note that $(A^N-p)$ is not in $\mathcal{W}_0$) \cite{gao2016}.

The dominant eigenvalue of $A^N$ is $pN+o(N)$ as $N\to\infty$, while all other eigenvalues are of order $N^{1/2}$, converging to a semicircle law after scaling \cite{ErdosI,ErdosII}.

It will be convenient to consider sequences of graphons where the density changes with $N$. To this end, we introduce a sequence $\kappa_N>0$, and for each $N$, consider the {\erg} on $N$ vertices with edge probability $\kappa_N$.

As $N\to\infty$, for the corresponding graphon $W^N$ we now have that $W_N/\kappa_N$ converges to the graphon $W(x,y)\equiv 1$ in $\|.\|_\square$ and $\|.\|_{\textrm{op.},p}$.

The following lemma enables approximation of the re-scaled empirical graphon operator $\frac{1}{\kappa^N}\W^N $ by the limiting graphon operator $\W$. The proof can be found in \cite{graphonL2conv}.

\begin{lemma}
\label{l:main}
Assume $W$ is blockwise continuous and $\frac{\log N}{\kappa^N N} \to 0$. Then 
\begin{align*}
& \left \| \frac{1}{\kappa^N}\W^N-\W \right \|_{\textrm{op.,}2} \to 0 \st
\end{align*}
where $st.$ denotes stochastic convergence.
\end{lemma} 

We remark that  $\|.\|_{op.,2}$ can be replaced by $\|.\|_\square$ in Lemma \ref{l:main} due to \eqref{eq:op2boxnormequiv} in the dense case when $ \kappa^N=1 $.




\section{Setup of the model}
\label{s:setup}

\subsection{Local density-dependent Markov processes on finite graphs}
\label{s:lddmp}

Using the setup of the previous section, assume we have a graphon kernel $W$ and a $\kappa_N$ sequence of densities, and $\{G^N:\,N=1,2,\dots\}$ denotes the (random) sample graphon from $W$ for each $N$.

Then we look to define a local density-dependent Markov process on $G^N$ for each $N$. Assume $N$ is fixed for now. The vertices of $G^N$ are $\{1,\dots,N\}$. $G^N$ itself and the process to be defined on $G^N$ are both random; $\pr_G$ will denote conditional probability with respect to $G^N$ (effectively making the graph $G^N$ fixed).

Each vertex of $G^N$ can be in one of the states from a finite set of states $\S$. The indicator $\xi_{i,s}^N(t)$ where $i \in \{1,\dots,N\}$ and $s\in\S$ denotes that vertex $i$ is in state $s \in \S$ at time $t$. Sometimes it will be convenient to map the collection of indicators according to
\begin{align}
\label{eq:xiN}
\xi^N(t,x):= \sum_{i=1}^N \xi_i^N(t) \1{I_i^N}(x),
\end{align}
where $I_i^N$ is according to \eqref{eq:IiNdef}.

$ \bar{\xi}_{s}^N(t)$ denotes the ratio of vertices in some state $s$ at time $t$. The vector $ \bar{\xi}^N(t)= \left( \bar{\xi}_{s}^N(t) \right)_{s \in S} $ will be referred to as \emph{state density} and can be calculated as 
\begin{align}
\label{eq:statedensityN}
& \bar{\xi}^N(t)= \frac{1}{N}\sum_{i=1}^N \xi_{i}^N(t)= \int_{0}^{1} \xi^N(t,x) \d x
\end{align}
Note that $\bar{\xi}^N(t) $ is always on the simplex
\begin{align*}
& \Delta := \left \{ \left. v \in \mathbb{R}^{\S } \right | \forall s \in \S \ v_s \geq 0, \sum_{s \in \S}v_s=1  \right \},
\end{align*}
and can be interpreted as an empirical probability vector.

The neighborhood of vertex $i$ can be described by
\begin{align*}
& \phi_{i,s}^N(t):= \frac{1}{N \kappa^N} \sum_{j=1}^N a_{ij}^N \xi_{j,s}^N(t),\qquad s\in \S.
\end{align*} 
(The normalization by $\kappa^N$ is global.) We will also use the vector notation $ \phi_i^N(t):= \left( \phi_{i,s}^N(t) \right)_{s \in \S} $ and call it the \emph{environment vector} of vertex $i$.

The scaling is justified by the identity
\begin{align}
\label{phinorm}
& \left \| \phi_i^N(t) \right \|= \sum_{s \in \S} \frac{1}{N \kappa^N} \sum_{j=1}^N a_{ij}^N \xi_{j,s}^N(t)=\frac{1}{N \kappa^N} \sum_{j=1}^N a_{ij}^N= \frac{d^N(i)}{N \kappa^N}.
\end{align}
Since $N \kappa^N $ is of the same order as the average degree, this means that the environment vectors are $O(1)$ for most vertices; however, $\phi_i^N(t)$ is not necessarily in the simplex $ \Delta $.

We can also rewrite the environment vectors using the graphon operator $\W^N$ corresponding to $G^N$ as
\begin{align}
\label{phiW}
& \phi_i^N(t)= \frac{1}{\kappa^N} \W^{N} \xi^N \left(t, \frac{i}{N} \right).
\end{align}

The evolution of the state of each vertex $i$ is according to a continuous time Markov chain with transition rate matrix $Q(\phi_i(t))=(q_{ss'}(\phi_i(t)))_{s,s'\in\S}$, where the transition rates are given $q_{ss'}: \mathbb{R}^{S} \to \mathbb{R} $ functions and the dependence on $\phi_i^N(t)$ corresponds to the \emph{local-density dependence} of the evolution: vertices are influenced only by their corresponding neighborhoods. Note that $q_{ss'}$ denotes the transition rate from state $s'$ to $s$. The diagonal elements are set to \linebreak$q_{ss}:= - \sum_{s' \neq s} q_{s's} $.

The transition rate functions are assumed to be nonnegative, i.e.~for all $ 0 \leq \phi \in \mathbb{R}^{\S}  $ we have $q_{ss'}(\phi)\geq 0 $. For technical reasons the nonnegativity requirement is extended for negative inputs too: we introduce an auxiliary matrix valued function $\hat{Q} $ with elements $\hat{q}_{ss'}:= \left |q_{s's} \right |  $ for $s \neq s' $ and $\hat{q}_{ss}:=-\sum_{s' \neq s}\hat{q}_{s's} $. Note that $ \hat{Q} \big|_{\phi \geq 0}=Q \big|_{\phi \geq 0}. $

For regulatory purposes we assume throughout this paper that the rate functions $q_{ss'}$ are globally Lipschitz continous with some Lipschitz constant $ L_{ss'}$. From the inequality
\begin{align*}
\big| \left|q_{ss'}(\phi) \right |-\left|q_{ss'}(\psi) \right | \big| \leq \left |q_{ss'}(\phi)-q_{ss'}(\psi) \right | \leq L_{ss'} \left \|\phi -\psi \right \|
\end{align*}    
 it is clear that we can use the same Lipschitz constants for $\hat{q}_{ss'}$.

The norm in $ \mathbb{R}^{S} $ induces a norm on $\mathbb{R}^{\mathcal{S} \times \mathcal{S}} $ matrices which we will denote by $ \left \|Q \right \|:= \sup \left \{ \left \|Qv \right \| : v \in \mathbb{R}^{\S}, \| v \|=1 \right \} $. $Q$ inherits its Lipschitz continuity from its components with some constant $L_Q$. Note that $ \hat{Q} $ and $Q$ may be assumed to have the same Lipschitz constant. The notation $Q_{\max}$ will refer to $ Q_{\max}:=\max_{ \phi \in \Delta } \left \|Q(\phi) \right \| $.

\subsection{The PDE limit}\label{s:PDE_limit}

Consider the following partial differential equation: 
\begin{align}
\label{eq:PDE}
& \partial_t u(t,x)= Q\left(\W u(t,x) \right)u(t,x).
\end{align}
In Section \ref{s:wellposedness} we show that \eqref{eq:PDE} has a unique global solution and that $u(t,x) \in \Delta $ provided $u(0,x) \in \Delta $ and $u(0, \cdot) $ is continuous.

The main result of the paper is that the solution of \eqref{eq:PDE} is the mean-field limit of $\xi^N(t,x)$ as $N\to\infty$:
\begin{theorem}
	\label{t:main}
	
	Let $W$ and $u(0)$ be continuous. Also assume \linebreak$ \left \| \xi^N(0)-u(0) \right \|_{\I} \to 0 \st $ and $ \frac{\log N}{\kappa^N N} \to 0. $ Then for any $T>0$,
	\begin{align}
	\label{eq:thm1}
	& \lim_{N \to \infty} \sup_{0 \leq t \leq T} \left\| \xi^N(t)-u(t) \right \|_{\I}=0 \st
	\end{align}
where $st.$ denotes stochastic convergence.	
\end{theorem}

The norm $\|.\|_\I$ is used for technical reasons. It ensures convergence of macroscopic quantities of interest: any Borel measurable $B \subset [0,1] $ can be approximated with union of intervals, so
\begin{align*}
\lim_{N \to \infty}\sup_{0 \leq t \leq T} \left \| \int_{B} \xi^N(t,x)-u(t,x) \d x \right \| = 0 \st,
\end{align*}
and so for any continuous and bounded function $ \varphi:[0,1]\to\mathbb{R}$ we have
\begin{align*}
\lim_{N \to \infty}\max_{s\in\S} \sup_{0 \leq t \leq T} \left | \int_{0}^{1} \varphi(x) \left [ \xi_s^N(t,x)-u_s(t,x) \right]  \d x\right |= 0 \st
\end{align*}

As a special case of \eqref{eq:thm1} for the interval $[0,1]$, the state density $\bar{\xi}^N(t)$ can be approximated by $ \bar{u}(t)$ uniformly in $[0,T]$.

We also mention that at the expense of more cumbersome notation, one can generalize this result to the case when $W$ and $u(0)$ is only piecewise continuous. In that case one can approximate stochastic processes on stochastic block models too with the PDE which reduces to a system of ODEs.

We also remark that while Section \ref{s:lddmp} sets up $G^N$ as a sample sequence from the limit graphon $W$, this is not the only interpretation; in fact, any sequence of graphons works that satisfies the convergence criterion of Lemma \ref{l:main}.

While the use of $\kappa^N$ enables relaxing the density constraint, it does not carry over to truly sparse graphs (i.e.~when $N \kappa^N =O(1)$). We state the counterexample as a separate theorem.

\begin{theorem}
	\label{t:r=1}	
		
	Consider the sparse {\erg} with parameter $\lambda>0$, that is, $\kappa^N=\frac{\lambda}{N}$ and $W(x,y) \equiv 1$. For the SIS process, there exists a sequence of initial conditions $ u_I(0,x) \equiv \bar{u}_I(0)= e^{-\lambda} $ such that
	$$ \left \|\xi^N(0)-\bar{u}_I(0) \right \|_\I\to 0\quad\textrm{ as }N\to\infty \textrm{ st.},$$ while for any $T>0$ 
	\begin{align*}
	&  \sup_{0 \leq t \leq T} \left | \bar{\xi}^N_I(t)-\bar{u}_I(0)e^{-t} \right |  \to 0 \st
	\end{align*}
	and
	\begin{align*}
	& \sup_{0 \leq t \leq T} \left |\bar{u}_I(0)e^{-t}-\bar{u}_I(t) \right |>0.
	\end{align*}
	where $st.$ denotes stochastic convergence.
\end{theorem}
The proofs of Theorems \ref{t:main} and \ref{t:r=1} are discussed in Section \ref{s:convproof}, but here we present a heuristic argument for Theorem \ref{t:main}.

Let
$$u^N_{i,s}(t):=\pr_G \left( \xi_{i,s}^N(t)=1\right)= \E_G \left( \xi_{i,s}^N(t) \right)$$
where the subscript $G$ refers to conditional probability and expectation with respect to the graph $G^N$. We will also use the notation
\begin{align*}
u_i^N(t)&:= \left(u_{i,s}^N(t)\right)_{s \in \S},\\
u^N(t,x)&:= \sum_{i=1}^N u_i^N(t) \1{I_i^N}(x),
\end{align*}
analogous to \eqref{eq:xiN}.

From the law of total expectation for the Markovian transitions, we have
\begin{align}
\label{totalexpectation}
& \frac{\d}{\d t}u_i^N(t)= \E_G \left[ Q \left( \frac{1}{N\kappa^N}\sum_{j=1}^N a_{ij}^N \xi_j^N(t) \right)\xi_j^N(t) \right].
\end{align}
When $N\kappa^N$ is large, each vertex has many neighbors, and in the average $\frac{1}{N\kappa^N}\sum_{j=1}^N a_{ij}^N \xi_j^N(t)$, the fluctuations due to the Markov process are low, so we may approximate it with its expectation $ \frac{1}{N\kappa^N}\sum_{j=1}^N a_{ij}^N u_j^N(t) $.

We also approximate the edges $a_{ij}^N$ with their expectation $\kappa^N W\left( \frac{i}{N}, \frac{j}{N}\right) $. This is sometimes called the annealed graph approximation \cite{annealed}. By setting $x= \frac{i}{N} $, this results in 
\begin{align*}
& \frac{1}{N\kappa^N}\sum_{j=1}^N a_{ij}^N u_j^N(t)  \approx \frac{1}{N}\sum_{j=1}^N  W\left( \frac{i}{N}, \frac{j}{N}\right) u_j^N(t) \approx  \\
& \int_{0}^{1}W(x,y)u^N(t,y) \d y= \W u^N(t,x),
\end{align*}
and putting these approximations back to \eqref{totalexpectation} yields
\begin{align*}
&\partial_{t}u^N(t,x) \approx Q\left( \W u^N(t,x)  \right)u^N(t,x),
\end{align*} 
suggesting $u^N(t,x)$ is close to $u(t,x)$ for large $N$.

The special case of the complete graph can be obtained by choosing $W(x,y) \equiv 1 $  for which, using the notation
$$\bar{u}(t):= \int_{0}^{1} u(t,x) \d x,$$
\eqref{eq:PDE} simplifies to
\begin{align*}
& \frac{\d}{\d t}\bar{u}(t)=Q \left( \bar{u}(t)\right)\bar{u}(t),
\end{align*}
which is indeed the mean-field limit equation identified by Kurtz \cite{kurtz70}.

Another special case is the SIS process, a simple model of epidemic spreading.  The state space is $ \S=\left \{S,I \right \}$ where $S$ stands for susceptible and $I$ for infected.

The dynamic is as follows: each infected vertex becomes susceptible (recovers) with rate $1$, while each susceptible vertex $i$ becomes infected with rate proportional to the number of its infected neighbors, that is, with rate $\beta \phi_{i,I}^N(t)$. The recovery rate $1$ is not restrictive, and can be achieved with an appropriate re-scaling of time.

The transition matrix can be written as
\begin{align*}
Q\left(\phi_i^N(t)\right)=\left[ {\begin{array}{cc}
	-\beta \phi_{i,I}^N(t) & 1 \\
	 \beta \phi_{i,I}^N(t) &   -1 \\
	\end{array} } \right].
\end{align*}

Since there are just two states, it is enough to consider quantities regarding the infected vertices only.

For state $I$ in the SIS process, \eqref{eq:PDE} becomes
\begin{align}
\label{eq:PDESIS}
& \partial_t u_I(t,x)=-u_I(t,x)+ \beta \left(1-u_I(t,x) \right)\W u_I(t,x).
\end{align}
A linearized version of \eqref{eq:PDESIS} has been studied in \cite{graphonSISnoise} with additional noise term and in \cite{graphonSIScontrol} with control.

\section{Properties of the PDE}
\label{s:pde}

In this section, we examine the partial differential equation \eqref{eq:PDE}. Section \ref{s:wellposedness} contains the proof for the existence and some properties of its solution. Then results concerning the dynamical structure of the SIS model \eqref{eq:PDESIS} are stated in Section \ref{s:SIS}. Section \ref{s:ODEdiscr} provides an approximation of \eqref{eq:PDE} using a discretisation of the operator $\W$.

\subsection{Well-posedness and positivity}
\label{s:wellposedness}

In this section we prove the existence of the solutions of equation \eqref{eq:PDE} and their biologically reasonable behavior, i.e. they do not leave the simplex $\Delta$. 

Let us use the notation $C[0,1]$ for continuous functions mapping from $[0,1]$ to $\mathbb{R}^{\S}$, and let us equip it with the usual uniform norm $\| . \|_{max}$. Then equation \eqref{eq:PDE} can be thought of as a differential equation defined on the Hilbert space $C[0,1]$.

\begin{theorem}\label{t:exist}
Under the above assumptions, equation \eqref{eq:PDE} has a unique solution, for which $u(t,\cdot) \in C[0,1]$. Also, if $u(0,x) \in \Delta$, then $u(t,x) \in \Delta$ for all \linebreak$t \in [0, \infty)$. 
\end{theorem}

\begin{proof}
For technical reasons, first we modify equation $\eqref{eq:PDE} $ by replacing $Q $ with $ \hat{Q} $.  Since the operator $\hat{Q}$ has the Lipschitz property, and  operator $\W$ is bounded, it can be shown that the right-hand side of \eqref{eq:PDE} has also the Lipschitz property. Then, by the usual arguments (see e.g. \cite[Theorem 1.1 in Section 7.1.]{daleckiibook} or \cite[Point 107. in Chapter 5]{volterrabook}) a local solution of  equation \eqref{eq:PDE} exists.

Now we prove that our solution stays in the set $\Delta$. Let us suppose that for a given $x \in [0,1]$ and for a time $t_0$, the local solution starting from $u(0,x)$ exists for $t \in [0,t^*]$ (the previous part of the proof assures its existence).

Let $t \in [0, t^*]$. Now we define an auxiliary inhomogeneous  Markov jump process indexed by $x \in [0,1] $, with transition rates $\hat{Q}(\W u(t,x)) $. The process is well defined for $t \in [0,t^*] $ since $u(t,x) $ exists (and bounded) and the rates $\hat{q}_{ss'}(\W u(t,x))$ are guarantied to be non-negative for $s \neq s'$ even if $u(t,x) $ might have a negative component. This might not be the case for $Q,$ hence the need for $\hat{Q}. $ The corresponding Kolmogorov equation is
\begin{align}\label{kolm}
p_x'(t) = \hat{Q} \left( \W u(t,x) \right) p_x(t).
\end{align}

It is well known that the solutions of \eqref{kolm} are in the set $\Delta$ since $ p_{x}(t)$ is a probability vector, so our goal now is to prove that the functions $p_x(t)$ and $u(t,x)$ are the same. This fact is the result of the following calculations.
\begin{align*}
\| p_x(t) & - u(t,x) \|  \leq \\
\| p_x(0) - & u(0,x) \|  + \int_{0}^{t} \| \hat{Q}(\W u(\tau,x)) p(\tau) - \hat{Q}(\W u(\tau,x)) u(\tau,x)  \| \d \tau \leq \\
  \| p_x(0) & -  u(0,x) \|  + \int_{0}^{t} \| \hat{Q}(\W u(\tau,x))\| \; \| p(\tau) - u(\tau,x)  \| \d \tau \leq 
\end{align*}
Now we use the fact that the function $u(\tau,x)$ is continuous and defined on a compact interval, so it is bounded. Also, since $\W$ and $\hat{Q}$ are bounded operators, the term $\| \hat{Q}(\W u(\tau,x))\| $ is bounded from above by a constant $K$.
\begin{align*}
\leq \| p_x(0) -  u(0,x) \|  + K \int_{0}^{t}  \| p(\tau) - u(\tau,x)  \| \d \tau .
\end{align*}
Then, by using the Gronwall inequality, we get
\begin{align*}
\| p_x(t) - u(t,x) \| & \leq  \| p_x(0) -  u(0,x) \|  e^{Kt} . 
\end{align*}
Since $p_x(0) = u(0,x)$, we get that $p_x(t) = u(t,x)$, so our solution stays in the set $\Delta$.

Then, since our solution is bounded, we can use the usual arguments (see e.g. the proof of \cite[Proposition 2.9. (ii)]{SISdyn}) to prove that our solution is also global.

The last step is to notice $ \left. \hat{Q} \right |_{\phi \geq 0}=\left. Q \right |_{\phi \geq 0}, $ thus, $u(t,x) \geq 0 $ implies $u(t,x) $ also satisfies the original \eqref{eq:PDE} without the hat.
\end{proof}

\subsection{Results for the SIS dynamics}
\label{s:SIS}

Other than the invariance of the set $\Delta$, another important dynamical feature of our model is the asymptotic behavior of the solution. In this section we are only considering the special case of equation \eqref{eq:PDE}, namely the SIS process introduced in Section \ref{s:PDE_limit} which involves the examination of the equation \eqref{eq:PDESIS}.

In their work Diekmann, Heesterbeek and Metz \cite{R0def} examined the dynamical behavior by defining a next generation operator as
\begin{align}\label{NGO}
T(g)(x)=\beta \int_0^1 W(x,y) g(y) \d y. 
\end{align}
In \cite{SISdyn}, Delmas, Dronnier and Zitt considered the spectral radius of an operator similar to \eqref{NGO}, and proved the following theorem \cite[Theorem 1.5. parts (ii)-(iii)]{SISdyn}.

\begin{theorem}
\label{t:threshold}
Let us denote the spectral radius of operator \eqref{NGO} by $R_0$. Assume the connectivity property
\begin{align*}
 \int_{A \times A^{c}}W(x,y)\d x \d x >0
\end{align*}
for all $A $ such that $\mu(A)>0 $ and $\mu\left(A^c \right)>0 $.
\begin{itemize}
\item If $R_0 \leq 1$, then the disease dies out: for all $x \in [0,1]$,
\begin{align*}
\lim_{t \rightarrow \infty} u_I (t,x)=0.
\end{align*}

\item If $R_0 > 1$, then there exists a unique equilibrium $g^* : [0,1] \rightarrow [0,1]$ with nonzero integral. For all initial conditions $u_I(0,x)$ such that its integral is positive:
\begin{align*}
\int_0^1 u_I(0,x) \d x >0,
\end{align*}
the solution $u_I$ converges pointwise to $g^*$, i.e., for all $x \in [0,1]$: 
\begin{align*}
\lim_{t \rightarrow \infty} u_I (t,x)= g^*(x)
\end{align*}
If $u_I(0,x) =0$ almost everywhere, then the solution $u_I$ converges pointwise to $0$.
\end{itemize}
\end{theorem}

Note that $R_0$ can be written in the form
\begin{align*}
R_0 = \beta \lambda_1(\W)
\end{align*}
in which $\lambda_1(\W)$ is as defined before, i.e. the eigenvalue of operator $\W$ with the largest absolute value. Theorem \ref{t:threshold} states that depending on the value of $\beta$, the following two outcomes are possible:
\begin{itemize}
\item If $\beta \leq \dfrac{1}{\lambda_1(\W)}$, then the disease dies out after some time.

\item If $\beta > \dfrac{1}{\lambda_1(\W)}$, then the ratio of infected people tends to a constant value.
\end{itemize}
Corresponding statements have been known to hold for epidemic thresholds for individual-based mean-field (IBMF) models \cite{threshold2008} and $N$-intertwined mean-field approximation (NIMFA) \cite{simon2017NIMFA}. For general finite networks, only the inequality $\beta_c>\frac{1}{\lambda_1}$ is known, but numerical simulations indicate it to be a very good approximation \cite{satorras2015}. Interestingly, Theorem \ref{t:threshold} shows it holds with equality for the graphon case.

Another interesting question is the value of function $g^*(x)$. It can be seen that a closed form in the case of the general $\W$ operator does not exist. However, if we suppose that it is separable, i.e. $W(x,y)=\varphi(x)\varphi(y)$, then it can be expressed in an explicit form.

\begin{prop}
Suppose that the kernel $W(x,y)$ has the form $W(x,y)=\varphi(x)\varphi(y)$. Then the equilibrium solution $g^*$ has the form
\begin{align}\label{g_def}
g^*(x)= \dfrac{\beta \varphi(x) k}{1 + \beta \varphi(x) k}
\end{align}
in which $k \in \mathbb{R}$ is the solution of the following (implicit) equation:
\begin{align}\label{k_def}
1 = \int_0^1 \dfrac{\beta (\varphi(x))^2 }{1 + \beta \varphi(x) k} \d x
\end{align}
\end{prop}

\begin{proof}
To compute the equilibrium solution, we have to solve the equation
\begin{align*}
0= \partial_t u_I(t,x) = -u_I(t,x)+ \beta \left(1-u_I(t,x) \right)\int_0^1 W(x,y) u_I(t,y) dy.
\end{align*}
Or, in other words,
\begin{align*}
0= -g^*(x)+ \beta \left(1-g^*(x) \right) \varphi(x) \int_0^1 \varphi(y) g^*(y) \d y.
\end{align*}
Using the notation $k:=\int_0^1 \varphi(y) g^*(y) dy$, the above equation has the form:
\begin{align*}
g^*(x) = \beta k \varphi(x) - \beta k g^*(x) \varphi(x),
\end{align*}
from which we get
\begin{align*}
g^*(x) = \dfrac{\beta \varphi(x) k}{1 + \beta \varphi(x) k}
\end{align*}
Now if we multiply both sides by $\varphi(x)$ and integrate it, then we get
\begin{align*}
k = \int_0^1 \dfrac{\beta (\varphi(x))^2 k}{1 + \beta \varphi(x) k} \d x,
\end{align*}
which is the same as \eqref{k_def}.
\end{proof}
Note that in the form \eqref{g_def}, the constant $k$ is unknown and its computation might be hard depending on the form of $\varphi(x)$, and might involve iteration methods. Nevertheless, the form \eqref{g_def} still gives us useful information about the form of our equilibrium solution.

\subsection{ODE approximation via discretisation}
\label{s:ODEdiscr}

In this section we show that by discretizing equation \eqref{eq:PDE}, we approximate said PDE to arbitrary precision in the finite time horizon. This both enables numerical investigations of the system and provides a useful tool for later proofs.  

$W^{(M)} $ will denote the discretized version of $W$ with cell size $\frac{1}{M}$, that is, 
$$W^{(M)}\left( x,y\right)=\sum_{k,l=1}^M W\left( \frac{k}{M}, \frac{l}{M}\right) \1{I_k^M}(x)\1{I_l^M}(y). $$
The corresponding integral operator is $\W^{(M)} $. $\W$ will be approximated by $\W^{(M)}$ for large $M$; the corresponding error term is defined as
\begin{align}
\label{def:delta}
& \delta_{M} :=\sup_{(x,y) \in [0,1]^2} \left | W^{(M)}(x,y)-W(x,y)  \right |.
\end{align}
$W$ is continuous and thus uniformly continuous on $[0,1]^2$, so $\lim_{M\to\infty}\delta_M = 0$. If we also assume $W$ to be Lipschitz-continuous with Lipschitz-constant $L_W$, then $\delta_M \leq \frac{L_W}{M}$.

\begin{lemma}
\label{W^M approx W}
Assume $W$ is continuous. Then
\begin{align*}
\left \|\W^{(M)}-\W \right \|_{\textrm{op.},1} \leq \delta_M
\end{align*}
\end{lemma}

\begin{proof}(Lemma \ref{W^M approx W})	

Let $ \varphi$ be an $L^1([0,1])$ function. Then
\begin{align*}
&\left \|\left(\W^{(M)}-\W \right) \varphi \right \|_1= \int_{0}^{1} \left | \left(\W^{(M)}-\W \right) \varphi(x) \right | \d x \leq \\
&\int_{0}^{1}  \int_{0}^{1}\left |W^{(M)}(x,y)-W(x,y) \right | \cdot \left | \varphi(y)\right| \d y \d x \leq \\
&\delta_M \int_{0}^{1} |\varphi(y)| \d y= \delta_M \|\varphi \|_1 \Rightarrow \\
&  \left \|\W^{(M)}-\W   \right \|_{\textrm{op.},1} \leq \delta_M
\end{align*}
\end{proof}
We create an auxiliary system of $S\times M$ ODEs of the form 
\begin{align}
\label{eqv_k}
\frac{\d}{\d t}v_{k}^M(t)= Q\left( \frac{1}{M}\sum_{l=1}^M W \left( \frac{k}{M},\frac{l}{M}\right)v_{l}^M(t) \right)v_{k}^M(t),\quad 1\leq k\leq M,
\end{align} 
and also use the notation $v_k^M(t)= \left( v_{k,s}^M(t)\right)_{s \in \S}$. 

Local existence of the solution of system $\eqref{eqv_k} $ follows from the local-Lipschitz continuity of the right hand side. Also, with the same technique as used in Section \ref{s:wellposedness}, it extends to global existence and $v_i^N(t) \in \Delta $. 

Using the notation
$$v^M(t,x):= \sum_{k=1}^M v_k^M(t) \1{I_k^M}(x),$$
we can rewrite \eqref{eqv_k} as
\begin{align}
\label{eqv}
& \partial_{t}v^M(t,x)= Q\left( \W^{(M)} v^M(t,x)\right)v^M(t,x).
\end{align} 
	
The following lemma states that $v^M(t)$ is a good approximation of $u(t)$.

\begin{lemma}
\label{v^M approx u} For arbitrary $T>0$ we have
\begin{align*}
&  \sup_{0 \leq t \leq T} \left \|v^M(t)-u(t)  \right \|_1=O \left(  \left \| v^M(0)-u(0) \right \|_1 +  \left \|\W^{(M)} -\W \right \|_{\textrm{op.},1} \right).
\end{align*}
\end{lemma}	

Note that if we further assume $u(0,x)$ and $W$ are Lipschitz and choose either
$$ v^M_k(0)=\frac{1}{1/M}\int_{I_k^M} u(0,x)\d x \quad \textrm{or}\quad v_k^M(0)= u \left( 0, \frac{k}{M} \right),$$ the error term becomes $O \left( \frac{1}{M} \right) $.
	
\begin{proof} (Lemma \ref{v^M approx u})
Assume $0 \leq t \leq T$. Then
\begin{align*}
& \left \| v^M(t,x)-u(t,x) \right \| \leq  \left \| v^M(0,x)-u(0,x) \right \|\\
&\qquad +\int_{0}^{t} \left \| Q\left(\W^{(M)} v^M(\tau,x) \right)v^M(\tau,x)-Q\left(\W u(\tau,x) \right)u(\tau,x)\right \| \d \tau.
\end{align*}

We can further decompose the errors as
\begin{align*}
& \left \| Q\left(\W^{(M)} v^M(\tau,x) \right)v^M(\tau,x)-Q\left(\W u(\tau,x) \right)u(\tau,x)\right \| \leq \\
& \left \|Q\left(\W^{(M)} v^M(\tau,x)\right)-Q\left(\W u(\tau,x)\right) \right \| \cdot \left \|v^M(\tau) \right \|+\\
& \qquad \left \| Q\left(\W u(\tau,x) \right) \right \| \cdot \left \|v^M(\tau,x)-u(\tau,x) \right \| \leq \\
& L_Q \left \|\W^{(M)} v^M(\tau,x)-\W u(\tau,x) \right \|+Q_{\max}\left \|v^M(\tau,x)-u(\tau,x) \right \| 
\end{align*}
Note that in the upper bound $Q_{\max}$ we implicitly used $ \W u(t,x) \in \Delta $. 
Using these upper bounds, integrating with respect to $x$ yields
\begin{align*}
\left \| v^M(t)-u(t) \right \|_1 &\leq \left \| v^M(0)-u(0) \right \|_1 \\ 
& + \int_{0}^{t} L_Q \left \|\W^{(M)} v^M(\tau)-\W u(\tau) \right \|_1+Q_{\max}\left \|v^M(\tau)-u(\tau) \right \|_1  \d \tau
\end{align*}
Additional error decomposition gives rise to
\begin{align*}
& \left \|\W^{(M)} v^M(\tau)-\W u(\tau) \right \|_1 \leq\\
& \left \|\W^{(M)} -\W \right \|_{\textrm{op.},1} \cdot \left \| v^M(\tau) \right \|_1+\left \| W \right \|_{\textrm{op.},1} \cdot \left \| v^M(\tau)- u(\tau) \right \|_1 \leq \\
& \left \|\W^{(M)} -\W \right \|_{\textrm{op.},1} + \left \| v^M(\tau)- u(\tau) \right \|_1,
\end{align*}
hence
\begin{align*}
\left \| v^M(t)-u(t) \right \|_1 &\leq \left \| v^M(0)-u(0) \right \|_1 +L_Q T \left \|\W^{(M)} -\W \right \|_{\textrm{op.},1}\\ 
& + \left(L_Q+ Q_{\max} \right)\int_{0}^{t} \left \|v^M(\tau)-u(\tau) \right \|_1  \d \tau.
\end{align*}
The last step is the application of Gronwall's inequality.
\begin{align*}
&  \sup_{0 \leq t \leq T}\left \| v^M(t)-u(t) \right \|_1 \leq\\
&\quad\big(  \left \| v^M(0)-u(0) \right \|_1
+L_Q T \left \|\W^{(M)} -\W \right \|_{\textrm{op.},1} \big)e^{\left(L_Q+ Q_{\max} \right)T}.
\end{align*}
\end{proof}

\section{Convergence of the stochastic process}	
\label{s:convproof}

We use the discretization of $[0,1]$ into $M$ parts from Section \ref{s:ODEdiscr} for the stochastic process as well. The intervals $I_1^M, \dots , I_M^M$ will be referred to as \emph{boxes}. The heuristic idea is that choosing $M$ large enough ensures that most vertices within a box will have similar properties due to the continuity of $W$, thus a box can be treated homogeneously. On the other hand, each box contains roughly $N/M$ vertices, thus, for $N>\!\!>M$, the random fluctuation within a given box is negligible, and therefore we can describe the dynamics of the boxes via \eqref{eqv_k}, which is a good approximation of $ \eqref{eq:PDE} $ according to Lemma \ref{v^M approx u}.

This argument also hints at the order at which $N$ and $M$ should be increased. We will use fixed $M$ first with $N\to\infty$, then take the limit $M\to\infty$ after.    

From now on, until stated otherwise, $M$ is considered to be fixed.

Let $V_k^{N,M} $ denote the indices of the vertices in box $k$:
$$V_k^{N,M}=\{i:1\leq i\leq N, \frac{i}{N} \in I_k^M\}.$$
$V_k^{N,M}$ contains roughly $N/M$ vertices:
$$ \left | \ \left |V_k^{N,M} \right |-N/M \right | \leq 1 .$$

The analogue of state density \eqref{eq:statedensityN} for the boxes is
\begin{align*}
& \xi_{k}^{N,M}(t):= \frac{1}{N/M} \sum_{i \in V_k^{N,M}} \xi_{i}^N(t), \\
& \xi^{N,M}(t,x):=\sum_{k=1}^M \xi_{k}^{N,M}(t) \1{I_k^M}(x).
\end{align*}
$ \bigcup_{i \in V_{k}^{N,M}} I_i^N $ is close to $ I_k^M $. Fo the error $D_k^{N,M} :=I_k^M \oplus\bigcup_{i \in V_{k}^{N,M}} I_i^N$ we have the upper bound
\begin{align}
\label{intervalerror}
& \mu \Big(  D_k^{N,M} \Big) \leq \frac{2}{N},
\end{align}
where $ \mu $ is the Lebesgue measure and $ \oplus $ denotes symmetric difference.

Since $ \| \xi^N(t,x)  \|=1 $, \eqref{intervalerror} implies
\begin{align}
\label{eq:inttransform}
& \frac{1}{1/M} \int_{I_k^M} \xi^{N}(t,x) \d x = \xi_k^{N,M}(t)+O \left( \frac{1}{N}\right) .
\end{align}

We use Poisson representation (see e.g.~\cite{hht2014}) for the dynamics of the process. Since the vertices are not identical (as in Kurtz \cite{kurtz70} or \cite{hht2014}), we need to be cautious to handle them separately. Accordingly, let $\mathcal{N}_{ss',i}^N(t)$ denote a family of independent Poisson processes with rate $1$. $\mathcal{N}_{ss',i}^N(t)$ will correspond to the state transitions of  vertex $i$ from state $s'$ to $s$. The evolution of $\xi_{i,s}^N(t) $ can be formulated as 
\begin{align*}
\xi_{i,s}^N(t)=\xi_{i,s}^N(0)&+ \sum_{s' \neq s} \mathcal{N}_{ss',i}^N \left( \int_{0}^{t} q_{ss'} \left( \phi_{i}^N(\tau) \right) \xi_{i,s'}^N(\tau) \d \tau \right) \\
&-\sum_{s' \neq s} \mathcal{N}_{s's,i}^N \left(  \int_{0}^{t} q_{s's} \left( \phi_{i}^N(\tau) \right) \xi_{i,s}^N(\tau) \d \tau \right).
\end{align*}

Next we group the vertices according to boxes.
\begin{align}
\label{eq:lonform}
\begin{split}
\xi_{k,s}^{N,M}(t)=\xi_{k,s}^{N,M}(0)&+\frac{1}{N/M} \sum_{i \in V_k^{N,M}} \sum_{s' \neq s} \mathcal{N}_{ss',i}^N \left( \int_{0}^{t} q_{ss'} \left( \phi_{i}^N(\tau) \right) \xi_{i,s'}^N(\tau) \d \tau \right) \\
&-\frac{1}{N/M} \sum_{i \in V_k^{N,M}}\sum_{s' \neq s} \mathcal{N}_{s's,i}^N \left(  \int_{0}^{t} q_{s's} \left( \phi_{i}^N(\tau) \right) \xi_{i,s}^N(\tau) \d \tau \right).
\end{split}
\end{align} 

To separate the process into drift and fluctuation terms, we first compute conditional expectation according to the filtration $\mathcal{F}_t$:
\begin{align*}
&\E \left( \left. \mathcal{N}_{ss',i}^N \left( \int_{0}^{t} q_{ss'} \left( \phi_{i}^N(\tau) \right) \xi_{i,s'}^N(\tau) \d \tau \right)   \right | \mathcal{F}_t \right)=\int_{0}^{t} q_{ss'} \left( \phi_{i}^N(\tau) \right) \xi_{i,s'}^N(\tau) \d \tau,
\end{align*}
then define fluctuation terms as
\begin{align*}
U_{ss',k}^{N,M}(t):=& \frac{1}{N/M} \sum_{i \in V_k^{N,M}}  \mathcal{N}_{ss',i}^N \left( \int_{0}^{t} q_{ss'} \left( \phi_{i}^N(\tau) \right) \xi_{i,s'}^N(\tau) \d \tau \right)\\
-&\frac{1}{N/M} \sum_{i \in V_k^{N,M}} \int_{0}^{t} q_{ss'} \left( \phi_{i}^N(\tau) \right) \xi_{i,s'}^N(\tau) \d \tau\\
U_{s,k}^{N,M}(t):=& \sum_{s' \neq s} U_{ss',k}^{N,M}(t)-U_{s's,k}^{N,M}(t) \\
U_{k}^{N,M}(t):=& \left( U_{s,k}^{N,M}(t) \right)_{s \in \S}.
\end{align*}
With this notation, \eqref{eq:lonform} can be written as
\begin{align}
\label{decomposition1}
& \xi_k^{N,M}(t)=\xi_k^{N,M}(0)+U_k^{N,M}(t)+ \int_{0}^{t} \frac{1}{N/M} \sum_{i \in V_k^{N,M}} Q \left( \phi_i^N(\tau) \right) \xi_i^N(\tau) \d \tau.
\end{align}

We aim to show that the Poisson fluctuation terms are negligible for large $N$. In order to do that, for technical simplicity we introduce independent rate 1 Poisson processes $\mathcal{N}_{ss',k}^{N,M}(t)$ for each box, and note that
\begin{align*}
& \sum_{i \in V_k^{N,M}} \mathcal{N}_{ss',i}^N \left( \int_{0}^{t} q_{ss'} \left( \phi_{i}^N(\tau) \right) \xi_{i,s'}^N(\tau) \d \tau \right)\overset{d}{=}\\
&\qquad \mathcal{N}_{ss',k}^{N,M} \left( \sum_{i \in V_k^{N,M}} \int_{0}^{t} q_{ss'} \left( \phi_{i}^N(\tau) \right) \xi_{i,s'}^N(\tau) \d \tau \right) ,
\end{align*}
from which
\begin{align*}
U_{ss',k}^{N,M}(t) \overset{d}{=} &  \frac{1}{N/M}   \mathcal{N}_{ss',k}^{N,M} \left(N/M \int_{0}^{t} \frac{1}{N/M} \sum_{i \in V_k^{N,M}} q_{ss'} \left( \phi_{i}^N(\tau) \right) \xi_{i,s'}^N(\tau) \d \tau \right)\\
-&\frac{1}{N/M} \sum_{i \in V_k^{N,M}} \int_{0}^{t} q_{ss'} \left( \phi_{i}^N(\tau) \right) \xi_{i,s'}^N(\tau) \d \tau \ !
\end{align*}

Now we are ready to address the fluctuation terms.
\begin{lemma}
	\label{l:fluctuation}
	For any $M$ and $T>0$ fixed,
	\begin{align*}
	\lim_{N \to \infty} \sup_{0 \leq t \leq T}\max_{1 \leq k \leq M} \left \| U_k^{N,M}(t) \right \|=0 \st
	\end{align*}  
	where $st.$ denotes stochastic convergence.
\end{lemma}

\begin{proof}(Lemma \ref{l:fluctuation})
	
	It is enough to show that $ \sup_{0 \leq t \leq T} \left| U_{ss',k}^{N,M}(t)\right | \to 0 \st $ as $N \to \infty $.
	
	Assume first that there is a constant $C(T)$ such that for all $0 \leq t \leq T $ we have
	\begin{align}
	\label{eq:CT}
	& \frac{1}{N/M} \sum_{i \in V_{k}^{N,M}} \int_{0}^{t} q_{ss'} \left( \phi_{i}^N(\tau) \right) \xi_{i,s'}^N(\tau) \d \tau \leq C(T).
	\end{align}
Then according to Doob's inequality
\begin{align*}
& \pr \left( \sup_{0 \leq t \leq T} \left | U_{ss',k}^{N,M}(t) \right | \geq \varepsilon \right) \leq \pr \left(\sup_{0 \leq t \leq C(T)} \left | \frac{1}{N/M} \mathcal{N}_{ss',k}^{N,M}\left(N/M \cdot t\right)-t \right | \geq \varepsilon \right) \leq \\
& \frac{M^2}{N^2 \varepsilon ^2} \mathbb{D}^2 \left(\mathcal{N}_{ss',k}^{N,M}\left(N/M \cdot C(T)\right) \right)=\frac{MC(T)}{N \varepsilon ^2} \to 0. 
\end{align*}	
	
	Next we show that $C(T)= T \left(2 |q_{ss'}(0)|+2ML_{ss'} \right) $ satisfies \eqref{eq:CT} with high probability.
	\begin{align}
	\nonumber
	& \frac{1}{N/M} \sum_{i \in V_{k}^{N,M}} \int_{0}^{t} q_{ss'} \left( \phi_{i}^N(\tau) \right) \xi_{i,s'}^N(\tau) \d \tau \leq \\
	\nonumber
	& t \cdot |q_{ss'}(0)| \frac{\left|V_k^{N,M} \right |}{N/M}+\frac{L_{ss'}}{N/M} \sum_{i \in V_k^{N,M}} \int_{0}^{t} \left \| \phi_i^N(\tau) \right \| \d \tau \leq \\
	\label{eq:CTbound}
	& T \left( |q_{ss'}(0)| \frac{\left|V_k^{N,M} \right |}{N/M}+\frac{L_{ss'}}{N/M} \sum_{i \in V_k^{N,M}} \frac{d^N(i)}{N \kappa^N} \right)
	\end{align}
For $N\geq M$, we have
\begin{align}
\label{eq:NM2}
\frac{\left|V_k^{N,M} \right |}{N/M} \leq 1+ \frac{1}{N/M} \leq 2 
\end{align}
estimating the first term on the right hand side of \eqref{eq:CTbound}. For the second term,
	\begin{align*}
	& \frac{1}{N/M} \sum_{i \in V_k^{N,M}} \frac{d^N(i)}{N \kappa^N}= \frac{1}{N/M} \sum_{i \in V_k^{N,M}} \frac{1}{\kappa^N} \left(\W^N 1 \right) \left( \frac{i}{N} \right) \leq \\
	& \frac{1}{N/M} \sum_{i =1}^N \frac{1}{\kappa^N} \left(\W^N 1 \right) \left( \frac{i}{N} \right)=M \int_{0}^1 \frac{1}{\kappa^N} \W^N 1(x) \d x.
	\end{align*} 
	Replacing $ \frac{1}{\kappa^N} \W^N $ by $\W $ would result in
	\begin{align*}
	& \int_{0}^1  \W 1(x) \d x = \int_{0}^{1}\int_{0}^{1} W(x,y) \d y \d x  \leq 1.
	\end{align*}
	The corresponding error can be bounded from above by
	\begin{align*}
	&   \int_{0}^{1} \left | \frac{1}{\kappa^N} \W^N 1(x)-\W1(x) \right | \d x \leq \sqrt{ \int_{0}^{1}\left | \frac{1}{\kappa^N} \W^N 1(x)-\W1(x) \right |^2 \d x}= \\
	& \left \| \left[ \frac{1}{\kappa^N} \W^N-\W \right ]1 \right \|_2 \leq \left \| \frac{1}{\kappa^N} \W^N-\W  \right \|_{\textrm{op.}2} \cdot \left \| 1 \right \|_2 = \left \| \frac{1}{\kappa^N} \W^N-\W  \right \|_{\textrm{op.}2} \leq 1,
	\end{align*}
	where the last inequality holds for large $N$ with high probability according to Lemma \ref{l:main}. Overall, putting the estimates in the right hand side of \eqref{eq:CTbound} ensures that $C(T)= T \left(2 |q_{ss'}(0)|+2ML_{ss'} \right) $ is an appropriate choice (when $N$ is large enough).
\end{proof}

The next lemma relates $\xi_k^{N,M}(t)$ to $ v_k^M(t) $.

\begin{lemma}
\label{l:hard}
Let $u(0,x)$ and $W$ be continuous and assume
$$ \left \| \xi^N(0)-u(0) \right \|_\I \to 0 \st $$
Let $v_k^M(t) $ be the solution of \eqref{eqv_k} with initial condition $ v_k^M(0)= \frac{1}{1/M}\int_{I_k^M} u(0,x) \d x $.

Then for any $T>0$ and $ \varepsilon >0 $ there exists $K=K(T) >0$ such that
\begin{align*}
& \lim_{N \to \infty} \pr \left(  \sup_{0 \leq t \leq T} \left \| \xi^{N,M}(t)-v^M(t) \right \|_{1} \geq \varepsilon+K \delta_M  \right)=0 
\end{align*} 
\end{lemma}

\begin{proof}(Lemma \ref{l:hard})

Since both $ \xi^{N,M}(t,x) $ and $v^M(t,x) $ are constant over each of the intervals $I_1^M, \dots, I_M^M $, we can decompose the error term according to the boxes:
\begin{align*}
& \psi^N(t) := \left \| \xi^{N,M}(t)-v^M(t)  \right \|_1= \frac{1}{M}\sum_{k=1}^M \left \| \xi_{k}^{N,M}(t)-v_k^M(t) \right \|.
\end{align*}
First, we show that the error term $\psi^N(0) $ from the initial condition vanishes using \eqref{eq:inttransform} and $ v_k^M(0)= \frac{1}{1/M}\int_{0}^{1}u(0,x) \d x $.
\begin{align*}
&\psi^N(0)=\frac{1}{M}\sum_{k=1}^M \left \| \xi_{k}^{N,M}(0)-v_k^M(0) \right \|= \\
& \sum_{k=1}^{M} \left \| \int_{I_k^M} \xi^N(0,x)- u(0,x) \d x\right \|+O \left( \frac{1}{N}  \right) \leq \\
&M \left \| \xi^N(0)-u(0) \right \|_\I +O \left( \frac{1}{N}  \right) \to 0 \st
\end{align*}	

With the help of \eqref{phiW} we can rewrite \eqref{decomposition1} as
\begin{align*}
 \xi_k^{N,M}(t)=& \xi_k^{N,M}(0)+U_{k}^{N,M}(t)+ \\
 &\int_{0}^{t} \frac{1}{N/M}\sum_{i \in V_{k}^{N,M}}Q \left( \frac{1}{\kappa^N}\W^N \xi^N\left(\tau, \frac{i}{N} \right) \right)\xi^N_i\left(\tau\right)  \d \tau.
\end{align*}
Next we replace $ \frac{1}{\kappa^N}\W^N $ with $ \W^{(M)} $. Using $ \W^{(M)} $ has the advantage that
\begin{align*}
&\W^{(M)}\xi^N\left(\tau, \frac{i}{N} \right)= \sum_{l=1}^M W \left( \frac{k}{M}, \frac{l}{M} \right) \int_{I_l^M} \xi^N(\tau,x) \d x=\\
& \frac{1}{M}\sum_{l=1}^M W \left( \frac{k}{M}, \frac{l}{M} \right) \xi_k^{N,M}(t)+O\left( \frac{1}{N}\right)
\end{align*}
has the same value for all $ i \in V_k^{N,M} $, thus
\begin{align}
\nonumber
& \frac{1}{N/M}\sum_{i \in V_{k}^{N,M}}Q \left( \W^{(M)} \xi^N\left(\tau, \frac{i}{N} \right) \right)\xi^N_i\left(\tau \right)= \\
\nonumber
& Q \left( \W^{(M)} \xi^N\left(\tau, \frac{i}{N} \right) \right) \left [ \frac{1}{N/M}\sum_{i \in V_{k}^{N,M}} \xi^N_i\left(\tau \right) \right] = \\
\label{eq:after W^(M) replacement}
& Q \left( \W^{(M)} \xi^N\left(\tau, \frac{i}{N} \right) \right) \xi_{k}^{N,M}(t)= \\
\nonumber
& Q \left( \frac{1}{M}\sum_{l=1}^M W \left( \frac{k}{M}, \frac{l}{M} \right) \xi_{l}^{N,M}(\tau) \right)\xi_{k}^{N,M}(\tau)+O \left( \frac{1}{N} \right)= \\
\nonumber
& Q \left( \W^{(M)} \xi^{N,M}\left(\tau, \frac{k}{M} \right) \right) \xi_k^{N,M}\left(\tau \right)+O \left( \frac{1}{N} \right).
\end{align}

Next we estimate the corresponding error
\begin{align*}
&  \left \| \frac{1}{N/M}\sum_{i \in V_{k}^{N,M}} \left [Q \left( \frac{1}{\kappa^N}\W^N \xi^N\left(\tau, \frac{i}{N} \right) \right)-Q \left( \W^{(M)} \xi^N\left(\tau, \frac{i}{N} \right) \right) \right] \xi^N_i\left(\tau\right)  \right \| \leq \\
& \frac{L_Q}{N/M}\sum_{i \in V_{k}^{N,M}} \left \|  \left [\frac{1}{\kappa^N}\W^N - \W^{(M)} \right] \xi^N\left(\tau, \frac{i}{N} \right)  \right \|.
\end{align*}

Here, we want to substitute the sums with integrals. However, it does not work as smoothly as in \eqref{eq:inttransform} since  $ \frac{1}{\kappa^N}\W^N \xi^N\left(\tau, \frac{i}{N} \right) $ might not be uniformly bounded in $ N$ or $M$ so we have to be a bit more careful.

Recall the definition of $D_k^{N,M}$ and observe $\bigcup_{i \in V_k^{N,M}}I_i^N \subset I_k \cup D_k^{N,M}$.
\begin{align*}
& \frac{L_Q}{N/M}\sum_{i \in V_{k}^{N,M}} \left \|  \left [\frac{1}{\kappa^N}\W^N - \W^{(M)} \right] \xi^N\left(\tau, \frac{i}{N} \right)  \right \|= \\
& L_Q M \int_{ \bigcup_{i \in V_k^{N,M}}I_i^N} \left \|  \left [\frac{1}{\kappa^N}\W^N - \W^{(M)} \right] \xi^N\left(\tau, x \right)  \right \| \d x \leq \\
& L_Q M \int_{ I_k^M \cup D_{k}^{N,M}} \left \|  \left [\frac{1}{\kappa^N}\W^N - \W^{(M)} \right] \xi^N\left(\tau, x \right)  \right \| \d x= \\
& L_Q M \sum_{s \in \S} \int_{0}^{1}  \1{I_k^M \cup D_{k}^{N,M}}(x) \left | \left [\frac{1}{\kappa^N}\W^N - \W^{(M)} \right] \xi_s^N\left(\tau, x \right) \right | \d x \leq \\ 
&  L_Q M \sum_{s \in \S} \int_{0}^{1}  \1{I_k^M \cup D_{k}^{N,M}}(x) \left | \left [\frac{1}{\kappa^N}\W^N - \W \right] \xi_s^N\left(\tau, x \right)  \right | \d x \\
& + L_Q M \sum_{s \in \S} \int_{0}^{1}  \1{I_k^M \cup D_{k}^{N,M}}(x) \left | \left [\W^{(M)} - \W \right] \xi_s^N\left(\tau, x \right)  \right | \d x
\end{align*}
For the first term,
\begin{align*}
&  L_Q M \sum_{s \in \S} \int_{0}^{1}  \1{I_k^M \cup D_{k}^{N,M}}(x) \left | \left [\frac{1}{\kappa^N}\W^N - \W \right] \xi_s^N\left(\tau, x \right)  \right | \d x \leq \\
&  L_Q M \sum_{s \in \S} \int_{0}^{1} \left | \left [\frac{1}{\kappa^N}\W^N - \W \right] \xi_s^N\left(\tau, x \right) \right | \d x \leq \\
& L_Q M \sum_{s \in \S}  \sqrt{\int_{0}^{1} \left | \left [\frac{1}{\kappa^N}\W^N - \W \right] \xi_s^N\left(\tau, x \right)  \right |^2 \d x}= \\
& L_Q M \sum_{s \in \S} \left \| \left [\frac{1}{\kappa^N}\W^N - \W \right] \xi_s^N\left(\tau \right)  \right \|_{2} \leq \\
& L_Q M \sum_{s \in \S}  \left  \| \frac{1}{\kappa^N}\W^N - \W \right \|_{\textrm{op.},2}  \cdot \left \| \xi_s^N\left(\tau \right) \right \|_2 \leq \\
& L_Q M | \S | \left  \| \frac{1}{\kappa^N}\W^N - \W \right \|_{\textrm{op.},2}.
\end{align*}

As for the second term,
\begin{align*}
& \left | \left [\W^{(M)} - \W \right] \xi_s^N\left(\tau, x \right)  \right |   \leq \int_{0}^{1} | W^{M}(x,y)-W(x,y)| \xi_s^N (\tau,y) \d y \leq \\
& \delta_M\int_{0}^{1} \xi_s^N (\tau,y) \d y= \delta_M \bar{\xi}^N_{s}(t),
\end{align*}
so based on \eqref{intervalerror} we have
\begin{align*}
& L_Q M \sum_{s \in \S} \int_{0}^{1}  \1{I_k^M \cup D_{k}^{N,M}}(x) \left | \left [\W^{(M)} - \W \right] \xi_s^N\left(\tau, x \right)  \right | \d x \leq  \\
& L_Q M \delta_M \left( \int_{0}^{1}  \1{I_k^M \cup D_{k}^{N,M}}(x)  \d x \right) \underbrace{ \sum_{s \in \S} \bar{\xi}_{s}^N(t)}_{=1}=L_Q \delta_M+O \left( \frac{1}{N}\right)
\end{align*}


These bounds and \eqref{eq:after W^(M) replacement} yield
\begin{align*}
&\psi^N(t) \leq \psi^N(0)+\frac{1}{M} \sum_{k=1}^{M} \left \| U_k^{N,M}(t)\right \|+ TL_Q M | \S | \left  \| \frac{1}{\kappa^N}\W^N - \W \right \|_{\textrm{op.},2}+\\
&\qquad\qquad O \left( \frac{1}{N} \right)+TL_Q \delta_M+ \\
&\frac{1}{M}\sum_{k=1}^M \left \| \int_{0}^{t} Q \left( \W^{(M)} \xi^{N,M}\left(\tau, \frac{k}{M} \right) \right) \xi^{N,M}_k\left(\tau \right)- Q \left( \W^{(M)} v^M\left(\tau, \frac{k}{M} \right) \right) v^M_k\left(\tau \right) \d \tau \right \|
\end{align*}
To take care of the last term, we employ similar techniques shown in the proof of Lemma \ref{v^M approx u}, using \eqref{eq:NM2}.

\begin{align*}
& \frac{1}{M}\sum_{k=1}^M \left \|  Q \left( \W^{(M)} \xi^{N,M}\left(\tau, \frac{k}{M} \right) \right) \xi^{N,M}_k\left(\tau \right)- Q \left( \W^{(M)} v^M\left(\tau, \frac{k}{M} \right) \right) v^M_k\left(\tau \right)  \right \| \leq \\
& \frac{1}{M}\sum_{k=1}^M \left \|  Q \left( \W^{(M)} \xi^{N,M}\left(\tau, \frac{k}{M} \right) \right) - Q \left( \W^{(M)} v^M\left(\tau, \frac{k}{M} \right) \right)  \right \| \cdot \left \| \xi^{N,M}_k\left(\tau \right) \right \|+\\
& \frac{1}{M}\sum_{k=1}^M \left \| Q \left( \W^{(M)} v^M\left(\tau, \frac{k}{M} \right) \right) \right \| \cdot \left \|  \xi^{N,M}_k\left(\tau \right)-  v^M_k\left(\tau \right)  \right \| \leq \\
& \frac{2L_Q}{M}\sum_{k=1}^M \left \|   \W^{(M)} \xi^{N,M}\left(\tau, \frac{k}{M} \right)  -  \W^{(M)} v^M\left(\tau, \frac{k}{M} \right)   \right \|+ \\
& \frac{Q_{\max}}{M}\sum_{k=1}^M \left \|  \xi^{N,M}_k\left(\tau \right)-  v^M_k\left(\tau \right)  \right \| = \\
& \frac{2L_Q}{M}\sum_{k=1}^M \left \|   \frac{1}{M}\sum_{l=1}^M W \left( \frac{k}{M}, \frac{l}{M} \right) \left[ \xi_{l}^{N,M}(\tau)-v_{l}^{M}(\tau) \right ]  \right \|+Q_{\max} \psi^N(\tau) \leq \\
& \frac{2L_Q}{M}\sum_{k=1}^M    \frac{1}{M}\sum_{l=1}^M  \left \| \xi_{l}^{N,M}(\tau)-v_{l}^{M}(\tau) \right \|+Q_{\max} \psi^N(\tau)= \\
& \left(2L_Q+ Q_{\max}\right)\psi^N(\tau)=:L \psi^N(\tau)
\end{align*}

Using this bound we have
\begin{align*}
&\psi^N(t) \leq  \psi^N(0)+\frac{1}{M} \sum_{k=1}^{M} \left \| U_k^{N,M}(t)\right \|+ TL_Q M | \S | \left  \| \frac{1}{\kappa^N}\W^N - \W \right \|_{\textrm{op.},2}+\\
& O \left( \frac{1}{N} \right)+TL_Q \delta_M+L \int_{0}^{t} \psi^N(\tau) \d \tau,
\end{align*}
then, using Lemma \ref{l:main} and Lemma \ref{l:fluctuation} with Gronwall:
\begin{align*}
& \sup_{0 \leq t \leq T} \psi^N(t) \leq \left( \psi^N(0)+\frac{1}{M} \sum_{k=1}^{M} \sup_{0 \leq t \leq T} \left \| U_k^{N,M}(t)\right \|+  \right. \\
& \left. TL_Q M | \S | \left  \| \frac{1}{\kappa^N}\W^N - \W \right \|_{\textrm{op.},2} +O \left( \frac{1}{N} \right)+TL_Q\delta_M\right )e^{LT} \\
& \to TL_Q e^{LT} \delta_M =:K \delta_M \st
\end{align*}
\end{proof}

\begin{proof}(Theorem \ref{t:main}.)

We decompose the error into two parts.
\begin{align*}
& \sup_{0 \leq t \leq T} \left \| \xi^N(t)-u(t) \right \|_\I \leq  \sup_{0 \leq t \leq T} \left \| \xi^N(t)-v^M(t) \right \|_\I+\sup_{0 \leq t \leq T} \left \| v^M(t)-u(t) \right \|_\I
\end{align*}	

According to Lemma \ref{W^M approx W} and \ref{v^M approx u} for all $ \varepsilon >0 $ there is a large enough $M$ such that
\begin{align*}
& \sup_{0 \leq t \leq T} \left \| v^M(t)-u(t) \right \|_\I \leq \sup_{0 \leq t \leq T} \left \| v^M(t)-u(t) \right \|_1 =\\
&O \left( \left \| v^M(0)-u(0) \right \|_1+ \left \| \W^{(M)}-\W \right \|_{\textrm{op.},1} \right) \leq \varepsilon.
\end{align*} 
	
For an interval $ J \in \I$, let $H_J^{M} $ denote the values of $ k$ for which $I_k^M \subset J $. Then we have
$$ \mu( J / \bigcup_{k \in H_J^M}I_k^M  ) \leq \frac{2}{M}.$$

Using Lemma \ref{l:hard}, for large enough $N$ and $M$ we have
\begin{align*}
&  \sup_{0 \leq t \leq T} \left \| \xi^N(t)-v^M(t) \right \|_\I= \sup_{0 \leq t \leq T}  \sup_{J \in \I}\left \| \int_{J} \xi^N(t,x)-v^M(t,x) \d x \right \| \leq \\
&  \sup_{0 \leq t \leq T}  \sup_{J \in \I}\left \| \int_{\bigcup_{k \in H_J^M}I_k^M} \xi^N(t,x)-v^M(t,x) \d x \right \|+\frac{2}{M} \leq \\
&  \sup_{0 \leq t \leq T}  \sup_{J \in \I}\sum_{k \in H_{J}^M}\left \| \int_{I_k^M} \xi^N(t,x)-v^M(t,x) \d x \right \|+\frac{2}{M} \leq 
\end{align*}
\begin{align*}
&   \sup_{0 \leq t \leq T} \frac{1}{M} \sum_{k=1}^M \left \| \frac{1}{1/M} \int_{0}^{1} \xi^N(t,x) \d x -v_k^M(t) \right \|+\frac{2}{M}= \\
&   \sup_{0 \leq t \leq T} \frac{1}{M} \sum_{k=1}^M \left \| \xi_k^{N,M}(t) -v_k^M(t) \right \|+O \left( \frac{1}{N} \right) +\frac{2}{M}= \\
&   \sup_{0 \leq t \leq T} \left \| \xi^{N,M}(t)-v^M(t) \right \|_1+O \left( \frac{1}{N} \right)+\frac{2}{M} \leq \\
&\varepsilon+K\delta_M +O \left(  \frac{1}{N}\right)+\frac{2}{M} \leq 2 \varepsilon
\end{align*}
with high probability. Thus
\begin{align*}
\sup_{0 \leq t \leq T} \left \| \xi^N(t)-u(t) \right \|_{\I} \leq 3 \varepsilon
\end{align*}
for large enough $N$ with high probability.
\end{proof}

The rest of the section provides the proof for Theorem \ref{t:r=1}.

The main idea of the proof is that if only 0 degree vertices are infected initially, then no further infection will occur, leading to an exponential decay in the ratio of infected nodes, which is drastically different compared to the logistic growth described by the mean field limit. 

Let the initial conditions are $ \xi_{i,I}^N(0)=\1{d^N(i)=0} $ and $u(0,x) \equiv e^{-\lambda} $. 

\begin{lemma}
\label{l:initialcond}
\begin{align*}
& \left \| \xi^{N}(0)-u(0,x) \right \|_{\I} \to 0 \st
\end{align*}
\end{lemma}

\begin{proof} (Lemma \ref{l:initialcond})

It is enough to consider the $I$ components as
\begin{align*}
&  \left \| \xi^{N}(0)-u(0,x) \right \|_{\I}= 2 \sup_{J \in \I} \left | \int_{J} \xi_{I}^N(0,x)-e^{-\lambda} \d x \right |. 
\end{align*}

First assume the following is true for all fixed $M$ and $1 \leq k \leq M $:
\begin{align}
\label{eq:fixedM}
& \xi^{N,M}_{k,I}(0) \to e^{-\lambda} \st
\end{align}
Also recall $H_J^{M} $ refers to those $k$ for which $I_k^M \subset J $ and $ \mu ( J / \bigcup_{k \in H_J^M}I_k^M) \leq \frac{2}{M} $. Then, according to \eqref{eq:inttransform},
\begin{align*}
& \sup_{J \in \I} \left | \int_{J} \xi_{I}^N(0,x)-e^{-\lambda} \d x \right | \leq \sup_{J \in \I} \left | \int_{\bigcup\limits_{k \in H_J^M}I_k^M} \xi_{I}^N(0,x)-e^{-\lambda} \d x \right |+\frac{2}{M} \leq \\
& \sup_{J \in \I} \frac{1}{M}\sum_{k \in I_k^M} \left | \frac{1}{1/M}\int_{I_k^M} \xi_{I}^N(0,x) \d x-e^{-\lambda}  \right |+\frac{2}{M}  \leq \\
&   \frac{1}{M}\sum_{k =1}^M \left | \xi_{k,I}^{N,M}(0)-e^{-\lambda}  \right |+\frac{2}{M}+O \left( \frac{1}{N} \right) \to \frac{2}{M} \st
\end{align*}

For the expectation of $\xi_{k,I}^{N,M}(0)$ we have
\begin{align*}
&\E \left( \xi_{k,I}^{N,M}(0) \right)=\frac{1}{N/M}\sum_{i \in V_k^{N,M}} \E(\xi_{i,I}^{N}(0))=\frac{1}{N/M}\sum_{i \in V_k^{N,M}} \pr \left(d^N(i)=0 \right)= \\
& \frac{ \left | V_k^{N,M} \right |}{N/M}\pr \left(d^N(1)=0 \right)= \frac{ \left | V_k^{N,M} \right |}{N/M} \left(1-\frac{\lambda}{N}\right)^{N-1} \to e^{-\lambda},
\end{align*}
so it is sufficient to prove concentration. Chebishev's inequality provides
\begin{align*}
& \pr \left( \left | \xi_{k,I}^{N,M}(0)- \E \left( \xi_{k,I}^{N,M}(0) \right) \right | \geq \varepsilon \right)= \\
& \pr \left(  \left |\sum_{i \in V_k^{N,M}} \left[ \xi_{i,I}^N(0)-\E \left( \xi_{i,I}^N(0) \right)   \right ] \right  | \geq  \frac{\varepsilon N}{M}  \right) \leq \\
& \frac{M^2}{\varepsilon^2 N^2} \sum_{i \in V_k^{N,M}}\sum_{i \in V_k^{N,M}} \cov \left( \xi_{i,I}^N(0), \xi_{j,I}^N(0) \right)= \\
& \frac{M^2}{\varepsilon^2 N^2} \left |V_k^{N,M} \right |\mathbb{D}^2 \left( \xi_{1,I}^N(0) \right) +\frac{M^2}{\varepsilon^2 N^2} {\left |V_k^{N,M} \right | \choose 2 } \cov \left(\xi_{1,I}^N(0),\xi_{1,I}^N(0) \right)= \\
& O \left( \frac{1}{\varepsilon^2 N} \right) \to 0, 
\end{align*}
where we used that $ \xi_{1,I}^N(0), \dots , \xi_{N,I}^N(0) $ have the same distribution and 
\begin{align*}
& \cov \left(\xi_{1,I}^N(0),\xi_{1,I}^N(0) \right)= \E \left(\xi_{1,I}^N(0) \xi_{2,I}^N(0) \right)-\E \left(\xi_{1,I}^N(0) \right)\E \left(\xi_{2,I}^N(0) \right)= \\
& \left(1-\frac{\lambda}{N} \right)^{2N-3}-\left(1-\frac{\lambda}{N} \right)^{2N-2} \leq \frac{\lambda}{N}.
\end{align*}
\end{proof}
Next we show that the state density of the Markov process differs from that of the PDE \eqref{eq:PDE}.

\begin{proof}(Theorem \ref{t:r=1})

The main idea is to compute the derivative of the state density at $t=0$  for both the Markov process and the PDE \eqref{eq:PDE} and check that they are different.

Let us introduce $ \varphi(t):= \bar{u}(0)e^{-t} $. Since $ \xi_i^{N}(t) $ are conditionally independent given the initial conditions and $ \E \left( \xi_{i,I}^{N}(t) \right) \to \varphi(t) $, the law of large numbers implies $ \bar{\xi}_I^{N}(t) \to \varphi(t) \st $ for all $t \geq 0$. 

This can be updated to uniform convergence due to the monotonicity of $ \bar{\xi}_I^{N}(t) $ and $\varphi(t)$. Let $T'$ be such that $ \varphi(T') \leq \varepsilon $. Then we can divide $[0,T'] $ into $0=t_0 < \dots <t_m =T' $ such that for all $1 \leq i \leq m $ we have $ |\varphi(t_i)-\varphi(t_{i-1})| \leq \varepsilon $.

If $t>T'$, then 
\begin{align*}
& \bar{\xi}_I^{N}(t)-\varphi(t) \leq \bar{\xi}_I^{N}(T')=\bar{\xi}_I^{N}(T')-\varphi(T')+\varphi(T') \leq \\
&\bar{\xi}_I^{N}(T')-\varphi(T')+ \varepsilon \leq  2\varepsilon,  \\
& \bar{\xi}_I^{N}(t)-\varphi(t) \geq -\varphi(T') \geq -\varepsilon.
\end{align*}
for large enough $N$ with high probability. 

Similarly, if $t_{i-1} \leq t <t_i $ for some $1 \leq i \leq m $, then
\begin{align*}
& \bar{\xi}_I^{N}(t)-\varphi(t) \leq \bar{\xi}_I^{N}(t_{i-1})-\varphi(t_i)=\bar{\xi}_I^{N}(t_{i-1})-\varphi(t_{i-1})+\varphi(t_{i-1})-\varphi(t_{i}) \leq \\
&\bar{\xi}_I^{N}(t_{i-1})-\varphi(t_{i-1})+ \varepsilon  \leq 2 \varepsilon , \\
& \bar{\xi}_I^{N}(t)-\varphi(t) \geq \bar{\xi}_I^{N}(t_{i})-\varphi(t_{i-1})=\bar{\xi}_I^{N}(t_{i})-\varphi(t_i)+\varphi(t_i)-\varphi(t_{i-1}) \geq \\
& \bar{\xi}_I^{N}(t_{i})-\varphi(t_i)-\varepsilon \geq -2\varepsilon .
\end{align*}
Thus for large $N$ one has $ \left | \bar{\xi}_I^{N}(t)-\varphi(t) \right | \leq 2 \varepsilon$ for all $t \geq 0 $ with high probability, implying
\begin{align*}
 \sup_{0 \leq t } \left | \bar{\xi}^N_{I}(t)-\varphi(t) \right | \to 0 \st
\end{align*}

On the other hand $ \sup_{0 \leq t \leq T} \left | \varphi(t)-\bar{u}_{I}(t) \right |>0 $ since 
$$ \frac{\d}{\d t} \varphi(0)=-\bar{u}_I(0) \neq -\bar{u}_I(0)+\beta \bar{u}_I(0) \left(1-\bar{u}_I(0) \right)= \frac{\d}{\d t} \bar{u}_I(0).  $$

\end{proof}

\section{Conclusion}
\label{s:concl}

We have presented here a detailed setup for local-density dependent Markov processes where the transition rates of the vertices are influenced by the states of their neighbors, identified the mean-field limit process on a graphon and provided rigorous proof of convergence under mild density conditions. We also examined the special case of the SIS process where we computed the epidemic threshold. We also showed that the mean-field limit does not hold when the density condition is violated and the average degree remains $O(1)$.

The question of identifying the mean-field limit when the average degree is $O(1)$ is still open. We expect the limit process to depend highly on both the structure of the network and the Markov process; however, any such result may require tools vastly different from graphons. This is subject to future work.

\bibliographystyle{abbrv}
\bibliography{mf}

\end{document}